\documentclass{article}

\usepackage{amsmath}
\usepackage{amssymb}
\usepackage{amsthm}
\usepackage[T1]{fontenc}
\usepackage{geometry}
\usepackage{hyperref}
\usepackage[all]{xy}

\newtheorem{thm}{Theorem}
\newtheorem*{thm**}{Theorem}

\theoremstyle{definition}
\newtheorem*{rmk}{Remark}

\DeclareMathOperator{\Ext}{Ext}
\DeclareMathOperator{\im}{im}

\title{A proof of the generalized geometric boundary theorem using filtered spectra}
\author{Sihao Ma}
\date{}

\begin{document}
\maketitle

\begin{abstract}
In \cite[Lem.~A.4.1]{good}, Behrens generalized the classical geometric boundary theorem \cite[Thm.~2.3.4]{rav86}. In this article, we will reformulate \cite[Lem.~A.4.1]{good} to fix a mistake made by Behrens, and prove it using the language of filtered spectra.
\end{abstract}

\section{Introduction}
The geometric boundary theorem was first introduced in \cite[Thm.~1.7]{jmwz} in the analysis of the generalized Adams spectral sequences, where they showed that the boundary homomorphisms between $\Ext$ groups have geometric interpretations.
\begin{thm**}[Johnson\textendash Miller\textendash Wilson\textendash Zahler]\label{convergence}
Let $E$ be a homotopy associative ring spectrum with unit such that $E_*$ is commutative and $E_*E$ is flat over $E_*$. Let $W\xrightarrow{f}X\xrightarrow{g}Y\xrightarrow{h}\Sigma W$ be a cofiber sequence of finite spectra with $E_*(h)=0$. If $\bar{x}\in\Ext_{E_*E}(E_*,E_*Y)$ converges to $x\in\pi_*Y$, then $\delta(\bar{x})$ converges to $h_*(x)\in\pi_*W$, where $\delta$ is the boundary homomorphism
\[\Ext_{E_*E}^{*,*}(E_*,E_*Y)\to\Ext_{E_*E}^{*+1,*}(E_*,E_*W).\]
\end{thm**}
Furthermore, Bruner showed in \cite{bruner} that the map $h$ induces not only the boundary homomorphism on the $E_2$-pages, but also a map between spectral sequences. It is summarized by Ravenel in \cite[Thm.~2.3.4]{rav86}.

Still, the naturality of spectral sequences and the classical geometric boundary theorem do not reveal all the information encoded in the fiber sequence. The differentials in the spectral sequences can be related in a more complicated way. Behrens generalized the classical geometric boundary theorem and proved it in \cite[Lem.~A.4.1]{good}. Heard, Li, and Shi also proved another generalized version of the classical geometric boundary theorem in \cite[Thm.~A.5]{hls}, which is the converse proposition of a special case of \cite[Lem.~A.4.1]{good}. Note that both of these two generalizations not only hold for the generalized Adams spectral sequences for fiber sequences, but also hold for spectral sequences for fiber sequences with compatible filtration towers.

In \cite[Lem.~A.4.1]{good}, Behrens claimed that there are five cases, among which one occurs. However, there is a mistake in the proof. In his discussion of the complement of Case (1), he claimed that there exists a lift of $y$ to $\tilde{y}$. In fact, this is equivalent to $p_*y'=0$ in $E^1_{t,\mu-\alpha}(Z)$, which is not guaranteed by the complement of Case (1): $[p_*y']=0$ in $E^{\alpha}_{t,\mu-\alpha}(Z)$. This might cause that none of the five cases occurs. Besides, the proof is hard to follow as there are many ad hoc constructions with complicated notations.

To fix these problems, in this article, we will reformulate \cite[Lem.~A.4.1]{good} by separating the five cases into five theorems and adding concrete conditions for them to happen. The language of filtered spectra comes in naturally when we state the results, and hence will also be used in our proof. Another benefit using filtered spectra is that we can get rid of the multiple complicated ad hoc constructions in Behrens' proof. The simplicity arising from the filtered spectra can be expected, since the bigraded homotopy groups of filtered spectra carry the information of differentials in different pages of a spectral sequence at the same time, which would be convenient when dealing with phenomena that happen in different pages, as in the generalized geometric boundary theorem.

\textbf{Conventions:} All diagrams in this article are considered to be diagrams in the homotopy category of filtered spectra $h\left(Sp^{\mathbb{Z}^{op}}\right)$ unless otherwise stated. The dashed arrows will denote the differentials in the spectral sequences. The double arrows will indicate that there is a spectral sequence from the homotopy groups of the source of double arrows converging to the homotopy groups of the target of double arrows. The arrows with the symbol $(-)$ will indicate that they are the inverses of the conventional maps with respect to the group structure of hom-groups of $h\left(Sp^{\mathbb{Z}^{op}}\right)$. The symbols on the sides of the diagram will indicate elements in the homotopy groups of the corresponding objects that are compatible under the diagram.

For an element $x$ in the $E^1$-page of a spectral sequence, if it does not support non-trivial differentials before the $E^{\alpha}$-page, we will use $[x]$ to denote the class $x$ represents in the $E^{\alpha}$-page.

\textbf{Acknowledgment:} The author would like to thank Mark Behrens for helpful conversations.

\section{An Introduction to filtered spectra}
In this section, we provide a brief introduction to filtered spectra following \cite{gikr}, and show some properties to be used later. We would like to mention that Pstr\k{a}gowski's work \cite{pst} on synthetic spectra provides another model which would also be beneficial when we restrict our attention to the generalized Adams spectral sequences for fiber sequences. By \cite[Cor.~6.1]{tl}, the category of $E$-synthetic spectra is equivalent to the category of modules over a filtered ring spectrum, which is a subcategory of the category of filtered spectra.

Let $Sp$ be the $\infty$-category of spectra. The category of filtered spectra is $Sp^{\mathbb{Z}^{op}}$, and a filtered spectrum can be denoted by
\[X=\cdots\to X_1\to X_0\to X_{-1}\to\cdots,\]
where $X_i\in Sp$. Note that homotopy limits and homotopy colimits are computed pointwise in this diagram category.

In $Sp^{\mathbb{Z}^{op}}$, there is a family of bigraded spheres
\[S^{t,\mu}=\cdots\to *\to *\to S^t\xrightarrow{\textup{id}}S^t\to\cdots,\]
where $S^{t,\mu}_{\alpha}=*$ for $\alpha>\mu$, and $S^{t,\mu}_{\alpha}=S^t$ for $\alpha\le\mu$. Then the bigraded homotopy groups of a filtered spectrum $X$ are given by
\[\pi_{t,\mu}X=\left[S^{t,\mu},X\right]=\pi_tX_\mu.\]

Note that there is a class $\tau\in\pi_{0,-1}S^{0,0}$ represented by the map $S^{0,-1}\to S^{0,0}$:
\[
\xymatrix{
\cdots \ar[r] & \ast \ar[r] \ar[d] & \ast \ar[r] \ar[d] & S^0 \ar[r] \ar[d] & \cdots \\
\cdots \ar[r] & \ast \ar[r] & S^0 \ar[r] & S^0 \ar[r] & \cdots.
}
\]
For $\alpha\ge0$, let $C\tau^{\alpha}$ be the cofiber of the map $\tau^{\alpha}:S^{0,-\alpha}\to S^{0,0}$. Applying Verdier's octahedral axiom to $S^{0,-\alpha-\beta}\xrightarrow{\tau^{\beta}}S^{0,-\alpha}\xrightarrow{\tau^{\alpha}}S^{0,0}$, we have
\[
\xymatrix@!C{
S^{0,-\alpha-\beta} \ar[r] \ar[rd] & S^{0,-\alpha} \ar[rrr] \ar[d] & & & \Sigma^{0,-\alpha}C\tau^{\beta} \ar[lldd] \\
 & S^{0,0} \ar[rd] \ar[dd] & & & \\
 & & C\tau^{\alpha+\beta} \ar[ld] & & \\
 & C\tau^{\alpha}, & & & \\
}
\]
and thus the fiber sequence
\begin{equation}\label{tau}
\Sigma^{0,-\alpha}C\tau^{\beta}\to C\tau^{\alpha+\beta}\to C\tau^{\alpha}.
\end{equation}

A filtered spectrum $X$ induces a decreasing filtration on $\pi_t\varinjlim X_{\mu}$ given by
\[F^{\alpha}\pi_t\varinjlim X_{\mu}=\im\left(\pi_tX_\alpha\to\pi_t\varinjlim X_{\mu}\right).\]
This filtration induces a spectral sequence in $Sp$. Let $X[\tau^{-1}]$ denote the filtered spectrum
\[X[\tau^{-1}]=\varinjlim\left(X\xrightarrow{\tau}\Sigma^{0,1}X\xrightarrow{\tau}\Sigma^{0,2}X\xrightarrow{\tau}\cdots\right).\]
Then its homotopy groups $\pi_{t,\mu}X[\tau^{-1}]\cong\pi_t\varinjlim X_{\mu}$ is independent on $\mu$. It is straightforward that the filtration mentioned above is isomorphic to the filtration on $\pi_{t,\mu}X[\tau^{-1}]$ given by
\[F^{\alpha}\pi_{t,\mu}X[\tau^{-1}]=\im\left(\pi_{t,\alpha}X\to\pi_{t,\alpha}X[\tau^{-1}]\right).\]
In order to avoid subtlety in the convergence of spectral sequences, for all filtered spectra $X$ appeared in this article, we will assume that for any fixed $t$, $\pi_{t,\mu}X=0$ for $\mu\gg0$, and $\pi_{t,\mu}X\to\pi_{t,\mu}X[\tau^{-1}]$ is an isomorphism for $\mu\ll0$. Then the spectral sequence
\[E^1_{t,\mu}=\pi_tX_{\mu}/X_{\mu+1}\Rightarrow\pi_t\varinjlim X_{\mu}\]
induced by $X$ is isomorphic to the $\tau$-Bockstein spectral sequence
\[E^1_{t,\mu}=\pi_{t,\mu}X\wedge C\tau\Rightarrow\pi_{t,\mu}X[\tau^{-1}].\]
From this point of view, any spectral sequence induced by a decreasing filtration in $Sp$ is isomorphic to a $\tau$-Bockstein spectral sequence in $Sp^{\mathbb{Z}^{op}}$.

This is not the first place where the $\tau$-Bockstein spectral sequence is compared with other spectral sequences. For instance, the generalized Adams spectral sequence is shown to be isomorphic to the $\tau$-Bockstein spectral sequence in the proof of \cite[Prop.~A.16]{bhs}. Furthermore, an explicit description of differentials in the generalized Adams spectral sequence using the boundary map of (\ref{tau}) is given in \cite[Thm.~9.19]{bhs}, which can be generalized to our situation. For an element $x\in E^1_{t,\mu}$ that survives to the $E^{\alpha}$-page, its target under $d_\alpha$ can be computed by first taking a lift of $-x$ in $\pi_{t,\mu}X\wedge C\tau^{\alpha}$, then mapping to $\pi_{t-1,\mu+\alpha}X\wedge C\tau=E^1_{t-1,\mu+\alpha}$ through the boundary map in the fiber sequence (\ref{tau}).

We end this section by introducing a construction that plays a key role in the proof of our main results.

Suppose there is a fiber sequence of filtered spectra
\[\Sigma^{-1,0}Z\xrightarrow{j}X\xrightarrow{i}Y\xrightarrow{p}Z.\]
By the boundedness assumption we made, let $C\tau^{\infty}$ be $S^{0,0}$. For $\alpha,\alpha'\in\mathbb N\cup\{\infty\}$, let $Y_{\alpha',\alpha}$ be the fiber of the composite map
\[Z\wedge C\tau^{\alpha'+\alpha} \to \Sigma^{1,0}X\wedge C\tau^{\alpha'+\alpha} \to \Sigma^{1,0}X\wedge C\tau^{\alpha}.\]
In particular, we have $Y_{0,\alpha}\simeq Y\wedge C\tau^{\alpha}$, and $Y_{\alpha,0}\simeq Z\wedge C\tau^{\alpha}$. The object $Y_{\alpha',\alpha}$ can be thought as an intermediate object between $Y\wedge C\tau^{\alpha'+\alpha}$ and $Z\wedge C\tau^{\alpha'+\alpha}$ constructed from the filtration on $X$, which carries more information on how the spectral sequences are compatible with the fiber sequences.

For $0\le\beta\le\alpha'$, applying Verdier's octahedral axiom to
\[Z\wedge C\tau^{\alpha'+\alpha}\to\Sigma^{1,0}X\wedge C\tau^{\alpha+\beta}\to\Sigma^{1,0}X\wedge C\tau^{\alpha},\]
we obtain the fiber sequence
\begin{equation}\label{diag}
\Sigma^{0,-\alpha}X\wedge C\tau^{\beta}\to Y_{\alpha'-\beta,\alpha+\beta}\to Y_{\alpha',\alpha}.
\end{equation}

For $\beta'\ge0$, applying Verdier's octahedral axiom to
\[Z\wedge C\tau^{\alpha'+\beta'+\alpha}\to Z\wedge C\tau^{\alpha'+\alpha}\to \Sigma^{1,0}X\wedge C\tau^{\alpha},\]
we obtain the fiber sequence
\begin{equation}\label{hor}
\Sigma^{0,-\alpha-\alpha'}Y_{\beta',0}\to Y_{\alpha'+\beta',\alpha}\to Y_{\alpha',\alpha}.
\end{equation}

\section{Statement of the main results}\label{sec3}
In this section, we will give the statement of the generalized geometric boundary theorem, while the proofs will be postponed to Section \ref{sec4}.

Suppose there is a fiber sequence of filtered spectra
\begin{equation}\label{fib}
\Sigma^{-1,0}Z\xrightarrow{j}X\xrightarrow{i}Y\xrightarrow{p}Z,
\end{equation}
which induces maps between the $\tau$-Bockstein spectral sequences. Let $y\in E^1_{t,\mu}(Y)$ be an element which survives to the $E^{\alpha}$-page and supports a differential
\[d_{\alpha}^Y([y])=[y'].\]

\begin{thm}\label{1}
There is a differential 
\[d_{\alpha}^Z([p_*y])=[p_*y'].\]
\[
\xymatrix{
y' \ar[rr]^{p_*} & & p_*y' & \\
& y \ar[lu]^{d_\alpha^Y} \ar[rr]^{p_*} & & p_*y. \ar[lu]_{d_\alpha^Z}}
\]
\end{thm}

Let $y'\in\pi_{t-1,\mu+\alpha}Y_{0,1}$ be a representative of the class $[y']$ in the $E^1$-page. There are two cases making this differential be trivial: either $p_*y'=0$ in the $E^1$-page, or $p_*y'$ has been killed by a shorter differential in the previous pages. The rest of the theorems will focus on the former case.

Suppose $p_*y'=0$ in the $E^1$-page. Then there exists a lift $y_\alpha\in\pi_{t,\mu}Y_{0,\alpha}$ of $y$, which maps to $y'\in\pi_{t-1,\mu+\alpha}Y_{0,1}$, and maps to $0\in\pi_{t-1,\mu+\alpha}Y_{1,0}$. By the fiber sequence (\ref{hor}), we have a fiber sequence
\[Y_{1,\alpha}\to Y_{0,\alpha}\to \Sigma^{1,-\alpha}Y_{0,1}.\]
Therefore, $y_\alpha$ can be lifted to $y_{1,\alpha}\in\pi_{t,\mu}Y_{1,\alpha}$, and so can $y$. Theorem \ref{5} will deal with the case that $y\in\pi_{t,\mu}Y_{0,1}$ can be lifted along the map $Y_{\infty,\alpha}\to Y_{0,1}$, while Theorem \ref{2}, Theorem \ref{3}, and Theorem \ref{4} will deal with the case that $y$ cannot be lifted. Note that the condition that $y$ can be lifted along the map $Y_{\infty,\alpha}\to Y_{0,1}$ is a sufficient but not necessary condition for $p_*y$ to be a permanent cycle, which is equivalent to the condition that $y$ can be lifted along the map $Y_{\infty,1}\to Y_{0,1}$.

\begin{thm}\label{2}
Let $y\in E^1_{t,\mu}(Y)$ be an element that survives to the $E^{\alpha}$-page. Suppose that there exists a maximal $\alpha'\in\mathbb N_+$ such that $y$ can be lifted to $y_{\alpha',\alpha}\in\pi_{t,\mu}Y_{\alpha',\alpha}$. Then there exist elements 
\begin{align*}
x & \in E^1_{t-1,\mu+\alpha}(X), \\
z & \in E^1_{t-1,\mu+\alpha+\alpha'}(Z),
\end{align*}
where $[z]\ne0$ in the $E^{\alpha'+1}$-page, such that
\begin{align*}
& d_{\alpha}^Y([y])=[i_*x], \\
& d_{\alpha'}^X([x])=[j_*z], \\
& d_{\alpha+\alpha'}^Z([p_*y])=[z].
\end{align*}
\[\xymatrix{
j_*z & & & & z \ar[llll]_{j_*} & \\
 & x \ar[lu]^{d_{\alpha'}^X} \ar[r]^{i_*} & i_*x & & & & \\
 & & & y \ar[lu]^{d_{\alpha}^Y} \ar[rr]^{p_*} & & p_*y. \ar[luu]_{d_{\alpha+\alpha'}^Z} 
}\]
\end{thm}

Theorem \ref{3} and Theorem \ref{4} will deal with the case that $j_*z=0$ in Theorem \ref{2}.

\begin{thm}\label{3}
Let $y_{\alpha',\alpha}$ be as in Theorem \ref{2}, and $x_{\alpha'}$ be the image of $y_{\alpha',\alpha}$ under the map
\[\pi_{t,\mu}Y_{\alpha',\alpha}\to\pi_{t-1,\mu+\alpha}X\wedge C\tau^{\alpha'}\]
as constructed in the fiber sequence (\ref{diag}). Suppose that there exists a maximal $\alpha''\in\mathbb N_+$ such that $x_{\alpha'}$ can be lifted to an element in $\pi_{t-1,\mu+\alpha}X\wedge C\tau^{\alpha'+\alpha''}$. Then there exist elements
\begin{align*}
x & \in E^1_{t-1,\mu+\alpha}(X), \\
x' & \in E^1_{t-2,\mu+\alpha+\alpha'+\alpha''}(X), \\
y'' & \in E^1_{t-1,\mu+\alpha+\alpha'}(Y),
\end{align*}
where $[p_*y'']\ne0$ in the $E^{\alpha'+1}$-page and $[x']\ne0$ in the $E^{\alpha''+1}$-page, such that
\begin{align*}
& d_{\alpha}^Y([y])=[i_*x], \\
& d_{\alpha'+\alpha''}^X([x])=[x'], \\
& d_{\alpha''}^Y([y''])=[-i_*x'], \\
& d_{\alpha+\alpha'}^Z([p_*y])=[p_*y''].
\end{align*}
\[\xymatrix{
x' \ar[rr]^{-i_*} & & -i_*x' & & & & \\
 & & & y'' \ar[lu]_{d_{\alpha''}^Y} \ar[rr]^{p_*} & & p_*y'' & \\
 & x \ar[luu]^{d_{\alpha'+\alpha''}^X} \ar[rr]^{i_*} & & i_*x & & & \\
 & & & & y \ar[lu]^{d_{\alpha}^Y} \ar[rr]^{p_*} & & p_*y. \ar[luu]_{d_{\alpha+\alpha'}^Z}
}\]
\end{thm}

\begin{rmk}
If $\alpha=\alpha''=1$, the lower two rows and the upper two rows give the construction of the algebraic boundary homomorphisms when we deal with the generalized Adams spectral sequences, and Theorem \ref{3} reduces to a part of the classical geometric boundary theorem.
\end{rmk}

\begin{thm}\label{4}
Let $x_{\alpha'}$ be as in Theorem \ref{3}. Suppose that $x_{\alpha'}$ can be lifted to an element in $\pi_{t-1,\mu+\alpha}X$. Then there exist elements
\begin{align*}
x & \in E^1_{t-1,\mu+\alpha}(X), \\
y'' & \in E^1_{t-1,\mu+\alpha+\alpha'}(Y), \\
\bar x & \in \pi_{t-1,\mu}X[\tau^{-1}],
\end{align*}
where $[p_*y'']\ne0$ in the $E^{\alpha'+1}$-page, such that
\begin{align*}
& d_{\alpha}^Y([y])=[i_*x], \\
& \bar x\begin{cases}\text{is detected by }x, & \text{if }x\text{ is not the target of a differential in }E^*_{*,*}(X), \\ \text{is in the image of }\pi_{t-1,\mu+\alpha}X, & \text{if }x\text{ is the target of a differential in }E^*_{*,*}(X),\end{cases}\\
& i_*\bar x\begin{cases}\text{is detected by }-y'', & \text{if }y''\text{ is not the target of a differential in }E^*_{*,*}(Y), \\ \text{is in the image of }\pi_{t-1,\mu+\alpha+\alpha'}Y, & \text{if }y''\text{ is the target of a differential in }E^*_{*,*}(Y),\end{cases}\\
& d_{\alpha+\alpha'}^Z([p_*y])=[p_*y''].
\end{align*}
\[\xymatrix{
\bar x \ar[rr]^{-i_*} & & -i_*\bar x & & & \\
 & & y'' \ar@{=>}[u] \ar[rr]^{p_*} & & p_*y'' & \\
x \ar@{=>}[uu] \ar[rr]^{i_*} & & i_*x & & & \\
 & & & y \ar[lu]^{d_\alpha^Y} \ar[rr]^{p_*} & & p_*y. \ar[luu]_{d_{\alpha+\alpha'}^Z}
}\]
\end{thm}

\begin{thm}\label{5}
Let $y\in E^1_{t,\mu}(Y)$ be an element that survives to the $E^{\alpha}$-page. Suppose that $y$ can be lifted to $y_{\infty,\alpha}\in\pi_{t,\mu}Y_{\infty,\alpha}$. Then there exist elements
\begin{align*}
\bar z & \in \pi_{t,\mu}Z[\tau^{-1}], \\
x & \in E^1_{t-1,\mu+\alpha}(X), \\
\end{align*}
where $[x]\ne0$ in the $E^{\alpha}$-page, such that
\begin{align*}
& d_{\alpha}^Y([y])=[i_*x], \\
& \bar z\begin{cases}\text{is detected by }p_*y, & \text{if }p_*y\text{ is not the target of a differential in }E^*_{*,*}(Z), \\ \text{is in the image of }\pi_{t,\mu}Z, & \text{if }p_*y\text{ is the target of a differential in }E^*_{*,*}(Z),\end{cases}\\
& j_*\bar z\begin{cases}\text{is detected by }x, & \text{if }x\text{ is not the target of a differential in }E^*_{*,*}(X), \\ \text{is in the image of }\pi_{t-1,\mu+\alpha}X, & \text{if }x\text{ is the target of a differential in }E^*_{*,*}(X).\end{cases}\\
\end{align*}
\[\xymatrix{
j_*\bar z & & & & & \bar z \ar[lllll]_{j_*} \\
x \ar@{=>}[u] \ar[rr]^{i_*} & & i_*x & & & \\
 & & & y \ar[lu]^{d_\alpha^Y} \ar[rr]^{p_*} & & p_*y. \ar@{=>}[uu]
}\]
\end{thm}

\begin{rmk}
When we deal with the generalized Adams spectral sequences, if $\alpha=1$, Theorem \ref{5} will reduce to the other part of the classical geometric boundary theorem connecting the algebraic boundary homomorphism and the geometric map.
\end{rmk}

\section{Proofs of the main results}\label{sec4}
In this section, we will prove the five theorems stated in Section \ref{sec3}.
\begin{proof}[Proof of Theorem~\ref{1}]
By the construction of differentials in the $\tau$-Bockstein spectral sequence, we can form a commutative  diagram
\[
\xymatrix{
y & Y_{0,1} \ar[r]^{p} \ar@{-->}@/_2.5pc/[dd]_{d_{\alpha}^Y} & Y_{1,0} \ar@{-->}@/^2.5pc/[dd]^{d_{\alpha}^Z} & p_*y\\
y_{\alpha} & Y_{0,\alpha} \ar[r]^{p} \ar[u] \ar[d]^{(-)} & Y_{\alpha,0} \ar[u] \ar[d]_{(-)} & p_*y_{\alpha}\\
y' & \Sigma^{1,-\alpha}Y_{0,1} \ar[r]^{p} & \Sigma^{1,-\alpha}Y_{1,0}. & p_*y'
}
\]
We can see that $p_*y$ can be lifted to $p_*y_{\alpha}$, which maps to $p_*y'$. Therefore, there is a differential in the $E^{\alpha}$-page:
\[d_{\alpha}^Z([p_*y])=[p_*y'].\qedhere\]
\end{proof}

\begin{proof}[Proof of Theorem~\ref{2}]
By considering the images of $y_{\alpha',\alpha}$, we have the following commutative diagram
\[
\xymatrix{
j_*z & \Sigma^{2,-\alpha'-\alpha}X\wedge C\tau & & \Sigma^{1,-\alpha'-\alpha}Y_{1,0} \ar[ll]_{j} & z \\
x_{\alpha'} & \Sigma^{1,-\alpha}X\wedge C\tau^{\alpha'} \ar[u] \ar[d]_{(-)} & Y_{\alpha',\alpha} \ar[l] \ar[r] \ar[d] \ar[ur]^(0.4){(-)} & Y_{\alpha'+\alpha,0} \ar[u]^{(-)} \ar[ddd] & z_{\alpha'+\alpha} \\
x & \Sigma^{1,-\alpha}X\wedge C\tau \ar[d]_{i} \ar@{-->}@/^3.5pc/[uu]^{d_{\alpha'}^X} & Y_{1,\alpha} \ar[l]_(0.35){(-)} \ar[d] & & \\
i_*x & \Sigma^{1,-\alpha}Y_{0,1} & Y_{0,\alpha} \ar[l]_(0.4){(-)} \ar[d] & & \\
 & y & Y_{0,1} \ar[r]^{p} \ar@{-->}[lu]^{d_{\alpha}^Y} & Y_{1,0}. \ar@{-->}@/_3.5pc/[uuuu]_{d_{\alpha+\alpha'}^Z} & p_*y
}
\]
This implies that $d_{\alpha}^Y([y])=[i_*x]$, $d_{\alpha+\alpha'}^Z([p_*y])=[z]$, and $d_{\alpha'}^X([x])=[j_*z]$.

Now we will show that $[z]\ne0$ in the $E^{\alpha'+1}$-page. Otherwise, $z$ can be lifted to $z_{\alpha'}\in\pi_{t,\mu+\alpha}Y_{\alpha',0}$. Then we have the commutative diagram
\[
\xymatrix@!C=60pt{
& z & \Sigma^{1,-\alpha'-\alpha}Y_{1,0} & \Sigma^{0,-\alpha}Y_{\alpha',0} \ar[l]_(0.45){(-)} \ar[ld]^{\textrm{(\ref{hor})}} & z_{\alpha'} \\
& y_{\alpha',\alpha} & Y_{\alpha',\alpha} \ar[ld] \ar[u]^{(-)} & y'_{\alpha',\alpha} & \\
y_{\alpha} & Y_{0,\alpha} & Y_{\alpha'+1,\alpha},\ar[l] \ar[u]_{\textrm{(\ref{hor})}} & y_{\alpha'+1,\alpha} & \\
}
\]
where $y'_{\alpha',\alpha}$ is the image of $z_{\alpha'}$. Note that $y_{\alpha',\alpha}-y'_{\alpha',\alpha}$ maps to $z-z=0$ in $\pi_{t-1,\mu+\alpha'+\alpha}Y_{1,0}$, and thus can be lifted to $y_{\alpha'+1,\alpha}$, which is also a lift of $y$. Then we arrive at a contradiction to the maximality of $\alpha'$.
\end{proof}

\begin{rmk}
In fact, the non-triviality of $[z]$ in the $E^{\alpha'+1}$-page is as far as we can guarantee. Using the similar method, we can show that $[z]=0$ in the $E^{\alpha'+\beta}$-page ($1\le\beta\le\alpha$) if and only if $y$ can be lifted to an element in $\pi_{t,\mu}Y_{\alpha'+\beta,\alpha+1-\beta}$. 
\end{rmk}

\begin{proof}[Proof of Theorem~\ref{3}]
Extending the diagram
\[
\xymatrix{
\Sigma^{0,-\alpha}X\wedge C\tau^{\alpha'} \ar[r] \ar[d] & \Sigma^{1,-\alpha'-\alpha}X\wedge C\tau^{\alpha''} \ar[d] \\
Y_{0,\alpha'+\alpha} \ar[r] & \Sigma^{1,-\alpha'-\alpha}Y_{0,\alpha''}
}
\]
in $Sp^{\mathbb{Z}^{op}}$ to two maps between fiber sequences, we obtain the following diagram in $h\left(Sp^{\mathbb{Z}^{op}}\right)$:
\[
\xymatrix{
\Sigma^{0,-\alpha}X\wedge C\tau^{\alpha'} \ar[r] \ar[d] & \Sigma^{1,-\alpha'-\alpha}X\wedge C\tau^{\alpha''} \ar[r] \ar[d] & \Sigma^{1,-\alpha}X\wedge C\tau^{\alpha'+\alpha''} \ar[r]^{\textrm{(\ref{tau})}} \ar[d] & \Sigma^{1,-\alpha}X\wedge C\tau^{\alpha'} \ar[d] \\
Y_{0,\alpha'+\alpha} \ar[r] \ar[d] & \Sigma^{1,-\alpha'-\alpha}Y_{0,\alpha''} \ar[r] \ar[d] & \Sigma^{1,0}Y_{0,\alpha''+\alpha'+\alpha} \ar[r]^{\textrm{(\ref{tau})}} & \Sigma^{1,0}Y_{0,\alpha'+\alpha} \\
Y_{\alpha',\alpha} \ar[r] \ar[d]_{\textrm{(\ref{diag})}} & \Sigma^{1,-\alpha'-\alpha}Y_{\alpha'',0} \ar[d] \\
\Sigma^{1,-\alpha}X\wedge C\tau^{\alpha'} \ar[r] & \Sigma^{2,-\alpha'-\alpha}X\wedge C\tau^{\alpha''}.
}
\]
Since $x_{\alpha'}$ can be lifted to $\pi_{t-1,\mu+\alpha}X\wedge C\tau^{\alpha'+\alpha''}$, the image of $y_{\alpha',\alpha}$ in $\pi_{t-2,\mu+\alpha'+\alpha}X\wedge C\tau^{\alpha''}$ must be $0$. By \cite[Lem.~5.1]{borel}, we have the following commutative diagram
\[
\xymatrix{
x' & & \Sigma^{2,-\alpha''-\alpha'-\alpha}X\wedge C\tau \ar[r]^(0.52){-i} & \Sigma^{2,-\alpha''-\alpha'-\alpha}  Y_{0,1} & -i_*x' & \\
x_{\alpha'+\alpha''} & & \Sigma^{1,-\alpha}X\wedge C\tau^{\alpha'+\alpha''} \ar[r] \ar[u] \ar[ddd] & \Sigma^{1,0}Y_{0,\alpha''+\alpha'+\alpha} \ar[u]^{(-)} & & \\
 & & & \Sigma^{1,-\alpha'-\alpha}Y_{0,\alpha''} \ar[r] \ar[u] \ar[d] & \Sigma^{1,-\alpha'-\alpha}Y_{0,1} \ar[d]^{p} \ar@{-->}[luu]_{d_{\alpha''}^Y} & y'' \\
 & & & \Sigma^{1,-\alpha'-\alpha}Y_{\alpha'',0} \ar[r] & \Sigma^{1,-\alpha'-\alpha}Y_{1,0} & p_*y'' \\
x_{\alpha'} & & \Sigma^{1,-\alpha}X\wedge C\tau^{\alpha'} \ar[d]^{(-)} & Y_{\alpha',\alpha} \ar[l] \ar[r] \ar[d] \ar[u]^{(-)} & Y_{\alpha'+\alpha,0} \ar[u]^{(-)} \ar[ddd] & z_{\alpha'+\alpha} \\
x & & \Sigma^{1,-\alpha}X\wedge C\tau \ar[d]_{i} \ar@{-->}`l_u/4pt[l]`/4pt[uuuuu]^{d_{\alpha'+\alpha''}^X}[uuuuu] & Y_{1,\alpha} \ar[l]_(0.4){(-)} \ar[d] & & \\
 & i_*x & \Sigma^{1,-\alpha}Y_{0,1} & Y_{0,\alpha} \ar[l]_(0.44){(-)} \ar[d] & & \\
 & & y & Y_{0,1} \ar[r]^{p} \ar@{-->}[lu]^{d_{\alpha}^Y} & Y_{1,0}, \ar@{-->}@/_3.5pc/[uuuu]_{d_{\alpha+\alpha'}^Z} & p_*y
}
\]
from which the conclusion follows. We can prove that $[x']\ne0$ in the $E^{\alpha''+1}$-page in a similar way as in the proof of Theorem \ref{2}, using the fact that $\alpha''$ is maximal such that $x_{\alpha'}$ can be lifted to $\pi_{t-1,\mu+\alpha}X\wedge C\tau^{\alpha'+\alpha''}$.
\end{proof}

\begin{proof}[Proof of Theorem~\ref{4}]
Extending the diagram
\[
\xymatrix{
\Sigma^{0,-\alpha}X\wedge C\tau^{\alpha'} \ar[r] \ar[d] & \Sigma^{1,-\alpha'-\alpha}X\wedge C\tau^{\alpha''} \ar[d] \\
Y_{0,\alpha'+\alpha} \ar[r] & \Sigma^{1,-\alpha'-\alpha}Y
}
\]
in $Sp^{\mathbb{Z}^{op}}$ to two maps between fiber sequences, we obtain the following diagram in $h\left(Sp^{\mathbb{Z}^{op}}\right)$:
\[
\xymatrix{
\Sigma^{0,-\alpha}X\wedge C\tau^{\alpha'} \ar[r] \ar[d] & \Sigma^{1,-\alpha'-\alpha}X \ar[r] \ar[d] & \Sigma^{1,-\alpha}X \ar[r] \ar[d] & \Sigma^{1,-\alpha}X\wedge C\tau^{\alpha'} \ar[d] \\
Y_{0,\alpha'+\alpha} \ar[r] \ar[d] & \Sigma^{1,-\alpha'-\alpha}Y \ar[r] \ar[d] & \Sigma^{1,0}Y \ar[r] & \Sigma^{1,0}Y_{0,\alpha'+\alpha} \\
Y_{\alpha',\alpha} \ar[r] \ar[d]_{\textrm{(\ref{diag})}} & \Sigma^{1,-\alpha'-\alpha}Z \ar[d] \\
\Sigma^{1,-\alpha}X\wedge C\tau^{\alpha'} \ar[r] & \Sigma^{2,-\alpha'-\alpha}X.
}
\]
Since $x_{\alpha'}$ can be lifted to $\pi_{t-1,\mu+\alpha}X$, the image of $y_{\alpha',\alpha}$ in $\pi_{t-2,\mu+\alpha'+\alpha}X$ must be $0$. By \cite[Lem.~5.1]{borel}, we have the following commutative diagram
\[
\xymatrix{
\bar x & \Sigma^{1,0}X[\tau^{-1}] \ar[r]^{-i} & \Sigma^{1,0}Y[\tau^{-1}] & -i_*\bar x & \\
 & \Sigma^{1,-\alpha}X \ar[u]_{(-)} \ar[r] \ar[ddd] & \Sigma^{1,0}Y \ar[u] & & \\
 & & \Sigma^{1,-\alpha'-\alpha}Y \ar[r] \ar[u] \ar[d] & \Sigma^{1,-\alpha'-\alpha}Y_{0,1} \ar[d]^{p} \ar@{=>}[luu] & y'' \\
 & & \Sigma^{1,-\alpha'-\alpha}Z \ar[r] & \Sigma^{1,-\alpha'-\alpha}Y_{1,0} & p_*y'' \\
x_{\alpha'} & \Sigma^{1,-\alpha}X\wedge C\tau^{\alpha'} \ar[d]^{(-)} & Y_{\alpha',\alpha} \ar[l] \ar[r] \ar[d] \ar[u]^{(-)} & Y_{\alpha'+\alpha,0} \ar[u]^{(-)} \ar[ddd] & z_{\alpha'+\alpha} \\
x & \Sigma^{1,-\alpha}X\wedge C\tau \ar[d]_{i} \ar@{=>}@/^5pc/[uuuuu] & Y_{1,\alpha} \ar[l]_(0.36){(-)} \ar[d] & & \\
i_*x & \Sigma^{1,-\alpha}Y_{0,1} & Y_{0,\alpha} \ar[l]_(0.41){(-)} \ar[d] & & \\
 & y & Y_{0,1} \ar[r]^{p} \ar@{-->}[lu]^{d_{\alpha}^Y} & Y_{1,0}, \ar@{-->}@/_3.5pc/[uuuu]_{d_{\alpha+\alpha'}^Z} & p_*y
}
\]
from which the result follows.
\end{proof}

\begin{proof}[Proof of Theorem~\ref{5}]
By considering the images of $y_{\infty,\alpha}$, we have the following commutative diagram
\[
\xymatrix{
j_*\bar z & \Sigma^{1,0}X[\tau^{-1}] & & Z[\tau^{-1}] \ar[ll]_j & \bar z \\
 & \Sigma^{1,0}X \ar[u] & & Z \ar[u] \ar[ll]_{j} \ar[dddd] &  \\
 & \Sigma^{1,-\alpha}X \ar[u]_{(-)} \ar[d]^{(-)} & Y_{\infty,\alpha} \ar[l] \ar[d] \ar[ur] & & \\
x & \Sigma^{1,-\alpha}X\wedge C\tau \ar[d]_{i} \ar@{=>}@/^3pc/[uuu] & Y_{1,\alpha} \ar[l]_(0.32){(-)} \ar[d] & & \\
y' & \Sigma^{1,-\alpha}Y_{0,1} & Y_{0,\alpha} \ar[l]_(0.38){(-)} \ar[d] & & \\
 & y & Y_{0,1} \ar[r]^{p} \ar@{-->}[lu]^{d_{\alpha}^Y} & Y_{1,0}, \ar@{=>}@/_3pc/[uuuuu] & p_*y
}
\]
from which the result follows.
\end{proof}

\bibliographystyle{alpha}
\bibliography{gbd}

\begin{thebibliography}{JMWZ75}

\bibitem[Beh12]{good}
Mark Behrens.
\newblock The {G}oodwillie tower and the {EHP} sequence.
\newblock {\em Mem. Amer. Math. Soc.}, 218(1026):xii+90, 2012.

\bibitem[BHS23]{bhs}
Robert Burklund, Jeremy Hahn, and Andrew Senger.
\newblock On the boundaries of highly connected, almost closed manifolds.
\newblock {\em Acta Mathematica}, 231(2):205--344, 2023.

\bibitem[Bru78]{bruner}
Robert Bruner.
\newblock Algebraic and geometric connecting homomorphisms in the {A}dams
  spectral sequence.
\newblock In {\em Geometric applications of homotopy theory ({P}roc. {C}onf.,
  {E}vanston, {I}ll., 1977), {II}}, volume 658 of {\em Lecture Notes in Math.},
  pages 131--133. Springer, Berlin-New York, 1978.

\bibitem[GIKR22]{gikr}
Bogdan Gheorghe, Daniel Isaksen, Achim Krause, and Nicolas Ricka.
\newblock {$\mathbb{C}$-motivic modular forms}.
\newblock {\em Journal of the European Mathematical Society},
  24(10):3597--3628, 2022.

\bibitem[HLS21]{hls}
Drew Heard, Guchuan Li, and XiaoLin~Danny Shi.
\newblock Picard groups and duality for real {M}orava {$E$}-theories.
\newblock {\em Algebr. Geom. Topol.}, 21(6):2703--2760, 2021.

\bibitem[JMWZ75]{jmwz}
David~C. Johnson, Haynes~R. Miller, W.~Stephen Wilson, and Raphael~S. Zahler.
\newblock Boundary homomorphisms in the generalized {A}dams spectral sequence
  and the nontriviality of infinitely many {$\gamma_t$} in stable homotopy.
\newblock In {\em Conference on homotopy theory ({E}vanston, {I}ll., 1974)},
  volume~1 of {\em Notas Mat. Simpos.}, pages 47--63. Soc. Mat. Mexicana,
  M\'{e}xico, 1975.

\bibitem[Law]{tl}
Tyler Lawson.
\newblock Synthetic spectra are (usually) cellular.
\newblock {\em Available at https://arxiv.org/abs/2402.03257}.

\bibitem[Ma]{borel}
Sihao Ma.
\newblock {The Borel and genuine $C_2$-equivariant Adams spectral sequences}.
\newblock {\em Available at https://arxiv.org/abs/2208.12883}.

\bibitem[Pst23]{pst}
Piotr Pstr{\k{a}}gowski.
\newblock Synthetic spectra and the cellular motivic category.
\newblock {\em Inventiones mathematicae}, 232(2):553--681, 2023.

\bibitem[Rav86]{rav86}
Douglas~C. Ravenel.
\newblock {\em Complex cobordism and stable homotopy groups of spheres}, volume
  121 of {\em Pure and Applied Mathematics}.
\newblock Academic Press, Inc., Orlando, FL, 1986.

\end{thebibliography}
\end{document}